\documentclass[11pt, reqno]{amsart}

\usepackage{amsmath, amssymb, mathtools, ytableau, tikz, xcolor}

\newtheorem{theorem}{Theorem}[section]

\newtheorem{proposition}[theorem]{Proposition}

\theoremstyle{remark}

\numberwithin{equation}{section}

\makeatletter
\@namedef{subjclassname@2020}{%
  \textup{2020} Mathematics Subject Classification}
\makeatother

\begin{document}


\title{The Littlewood decomposition via colored Frobenius partitions}

\author[H. Cho]{Hyunsoo Cho}
\address{Institute of Mathematical Sciences, Ewha Womans University,  52 Ewhayeodae-gil, Seodaemun-gu, Seoul 03760, Republic of Korea}
\email{hyunsoo@ewha.ac.kr}

\author[E. Kim]{Eunmi Kim}
\address{Institute of Mathematical Sciences, Ewha Womans University,  52 Ewhayeodae-gil, Seodaemun-gu, Seoul 03760, Republic of Korea}
\email{ekim67@ewha.ac.kr; eunmi.kim67@gmail.com}

\author[A. J. Yee]{Ae Ja Yee}
\address{Department of Mathematics, The Pennsylvania State University, University Park, PA 16802, USA}
\email{yee@psu.edu}

\dedicatory{Dedicated to George Andrews and Bruce Berndt for their 85th birthdays}


\begin{abstract} 
The Littlewood decomposition for partitions is a well-known bijection between partitions and pairs of $t$-core and $t$-quotient partitions. This decomposition can be described in several ways, such as the $t$-abacus method of James or the biinfinite word method of Garvan, Kim, and Stanton. In a recent study, Frobenius partitions have proven to be a highly useful tool in dealing with partition statistics related to $t$-core partitions. Motivated by this study, in this paper, we present an alternative description of the Littlewood decomposition using Frobenius partitions. We also apply our approach to self-conjugate partitions and doubled distinct partitions, and give new characterizations of their $t$-cores and $t$-quotients. 
\end{abstract}

\keywords{Frobenius partitions, Littlewood decomposition, Partition hooks, self-conjugate partitions, doubled distinct partitions}

\subjclass[2020]{05A17, 11P81}

\maketitle


\section{introduction}

For a positive integer $n$, a partition $\lambda=(\lambda_1,\lambda_2,\dots,\lambda_\ell)$ of $n$ is a weakly decreasing sequence of positive integers whose sum is $n$. Each $\lambda_i$ is called a part of $\lambda$. The sum of parts is called the size of $\lambda$ and is denoted by $|\lambda|$. It is a convention to define the empty sequence $\emptyset$ to be the partition of $0$.
The Young diagram of a partition $\lambda$ is a finite collection of boxes arranged in left-justified rows with $\lambda_i$ boxes in the $i$th row. In Figure~\ref{fig:diagram}, the Young diagram of the partition $(8,7,7,4,4,2)$ is illustrated.

\begin{figure}[ht!]
	\centering
	\begin{ytableau}
		13&12&10&9&6&5&4&1\\
		11&10&8&7&4&3&2\\
		10&9&7&6&3&2&1\\
		6&5&3&2\\
		5&4&2&1\\
	2&1
	\end{ytableau}
	\caption{The Young diagram of the partition $(8,7,7,4,4,2)$ with its hook lengths}
	\label{fig:diagram}
\end{figure}

In the Young diagram of a partition $\lambda$, the box in the $i$th row and the $j$th column is labeled $(i,j)$. The {\it hook length} of the box $(i,j)$, denoted by $h_{i,j}(\lambda)$, is the number of boxes below and to the right of the box $(i,j)$ plus 1, the box itself. In Figure \ref{fig:diagram}, the number in each box is its hook length. For a positive integer $t$, a partition $\lambda$ is called a {\it $t$-core} partition if no hook lengths are divisible by $t$. 

In the theory of modular representation of symmetric groups, $t$-core partitions play an important role \cite{N}. 
It was also shown by Littlewood 
\cite{L} that 
a partition $\lambda$ can be uniquely decomposed as a $t$-core partition $\lambda^{(t)}$ and a $t$-multipartition $\big( \lambda_{(0)}, \lambda_{(1)}, \ldots, \lambda_{(t)} \big)$. These $t$-core partition and $t$-multipartition are called the $t$-core and $t$-quotient of $\lambda$, respectively. This one-to-one correspondence is known as the {\it Littlewood decomposition} of $\lambda$ at $t$. 

The Littlewood decomposition is one of the essential components in defining the partition cranks of Garvan, Kim, and Stanton~\cite{GKS}, which give a combinatorial account for the mod 5, 7, and 11 partition congruences of Ramanujan. In a series of papers by Berkovich and Garvan~\cite{BG0, BG1, BG2}, this decomposition idea has been adopted for new partition statistics, namely the BG-rank and the GBG-rank, to give another combinatorial account for Ramanujan’s mod 5 partition congruence and to study the arithmetic properties of the generating function for $t$-core partitions.

In addition to the original description of Littlewood \cite{L}, \cite[pp.12--13]{M}, the Littlewood decomposition can be described in several other ways, for instance, James~\cite{James, JK} proposed the $t$-abacus method, while Garvan, Kim, and Stanton~\cite{GKS} utilized $t$-residue diagrams and biinfinite words, which have a crucial role in the development of the BG-rank and its generalizations in \cite{BD, BG0, BG1, BG2}. However, in a recent study~\cite{CKLLYY}, Frobenius partitions have proven to be a highly useful tool in dealing with partition statistics related to $t$-core partitions. Motivated by this study, in this paper, we give another full account for the Littlewood decomposition using Frobenius partitions. 

It should be noted that Frobenius partitions were also used in \cite{BN,K} to further study $t$-core partitions. Kolitsch \cite{K} was the first who gave a characterization of $t$-core partitions in terms of Frobenius partitions by analyzing the $t$-residue diagram and biinfinite word method of Garvan, Kim, and Stanton. In \cite{BN}, Brunat and Nath provided a more thorough and extensive study of $t$-cores and quotients employing Frobenius partitions and the abacus method of James. 
However, our approach in this paper is completely independent of the $t$-abacus method and the biinfinite word method. Thus, our proofs are more straightforward and transparent than those in \cite{BN, K}. Our method can also be applied to derive other results, such as formulas for the number of $t$-hooks \cite{CKKY}.  

A {\it Frobenius partition} for a positive integer $n$ is a two-rowed array of the form 
\[
	\begin{pmatrix} 
		a_1 & a_2 & \cdots & a_s \\ 
		b_1 & b_2 & \cdots & b_s 
	\end{pmatrix}
\]
such that $a_1>a_2>\cdots >a_s\ge 0$, $b_1>b_2>\cdots >b_s\ge 0$, and
\[
	\sum_{i=1}^s (a_j+b_j+1)=n.
\]
There is a natural one-to-one correspondence between partitions and Frobenius partitions. Given a partition $\lambda$, let $s$ be the largest integer such that $\lambda_s - s \ge 0$; this $s$ represents the side length of the largest square that can fit within the Young diagram of $\lambda$, often referred to as the {\it Durfee square}. The corresponding Frobenius partition is then represented by the two-rowed array
\[
	\mathfrak{F}(\lambda)=
	\begin{pmatrix} 
		\lambda_1 - 1 & \lambda_2 - 2 & \cdots & \lambda_s - s \\
		\lambda'_1 - 1 & \lambda'_2 - 2 & \cdots & \lambda'_s - s 
	\end{pmatrix},
\]
where $\lambda'_j$ denotes the number of boxes in the $j$th column of the Young diagram of $\lambda$. This array $\mathfrak{F}(\lambda)$ satisfies the conditions for Frobenius partitions, and the process is reversible, ensuring a unique Frobenius partition for each partition $\lambda$. For example, the Frobenius partition of the partition $\lambda = (8, 7, 7, 4, 4, 2)$ is
\[
    \mathfrak{F}(\lambda)=
	\begin{pmatrix} 
		7 & 5 & 4 & 0 \\
		5 & 4 & 2\ & 1 
	\end{pmatrix}.
\]    

Equivalently, the correspondence between partitions and Frobenius partitions can be stated as follows:
\begin{equation}\label{eq:Frob}
	\sum_{\lambda} q^{|\lambda|} =[z^0] (-zq;q)_{\infty} (-1/z;q)_{\infty}, 
\end{equation}
where the sum on the left side is over all partitions. Here and throughout this paper, $[z^0]F(z)$ will denote the constant term of a series $F(z)$ in $z$, and the following $q$-Pochhammer symbol will be used:
\[
	(a;q)_{\infty}:=\prod_{k=0}^{\infty} (1-aq^k). 
\]
Upon an application of Jacobi's triple product identity (see Section~\ref{sec2} for the identity),
the right side of \eqref{eq:Frob} expands as
\[
	[z^0]\frac{1}{(q;q)_{\infty}}\sum_{n=-\infty}^{\infty} z^n q^{\binom{n+1}{2}},
\]
from which the partition generating function immediately follows, confirming the correspondence with Frobenius partitions. The main objective of this paper is to adopt this correspondence and present an alternative description of the Littlewood decomposition using Frobenius partitions.

We will assume $t$ to be a positive integer throughout this paper. Let $\mathcal{P}$ be the set of all partitions and $\mathcal{C}_t$ be the set of all $t$-core partitions. The next theorem is our main result.
\begin{theorem}\label{main2}
	There is a bijection $\varphi$ between $\mathcal{P}$ and $\mathcal{C}_t\times \mathcal{P}^t$ defined by 
   	\[
    	\varphi(\lambda)=\left( \lambda^{(t)}, (\lambda_{(0)},\lambda_{(1)},\dots,\lambda_{(t-1)})\right) 
    \]
    with $|\lambda|=|\lambda^{(t)}|+t\sum\limits_{j=0}^{t-1}|\lambda_{(j)}|$.
\end{theorem}

As an application of our theorem, we give a concise and elegant description of the $t$-core in terms of Frobenius partitions. 
This is closely related to a result of Brunat and Nath \cite[Lemma~3.19]{BN}.

\begin{theorem} \label{main1}
	A partition $\lambda$ is a $t$-core partition if and only if $\mathfrak{F}(\lambda)=\begin{pmatrix} 
		a_1 & a_2 & \cdots & a_s \\ 
		b_1 & b_2 & \cdots & b_s 
	\end{pmatrix}$ satisfies the following conditions: 
    \begin{enumerate}
        \item For any $i,j$, $a_i+b_j+1 \not\equiv 0 \pmod{t}$.
        \item If $a_i\ge t$, then $a_i-t$ must appear in the first row of $\mathfrak{F}(\lambda)$.
        \item If $b_j\ge t$, then $b_j-t$ must appear in the second row of $\mathfrak{F}(\lambda)$. 
    \end{enumerate}
\end{theorem}

As further applications, we will study the Littlewood decomposition for self-conjugate partitions and doubled distinct partitions, and present characterizations of their $t$-cores and $t$-quotients in terms of Frobenius partitions. 

The rest of this paper is organized as follows. In Section \ref{sec2}, we will give the definition of $t$-colored Frobenius partitions and a brief review of Jacobi's triple product identity and Wright's map, which will be used in later sections. In Section \ref{sec3}, our bijection $\varphi$ will be presented and Theorems \ref{main2} and \ref{main1} will be proved. Also, theorems on $t$-core partitions and hooks of length $t$ will be given. In Section \ref{sec4}, self-conjugate partitions and doubled distinct partitions will be discussed.

\section{Preliminaries}\label{sec2}

\subsection{$t$-Colored Frobenius partitions}
Andrews~\cite{GEA} generalized Frobenius partitions by relaxing the condition on the entries of a Frobenius partition. He gave a thorough study to two types of generalizations, one of which is $t$-colored Frobenius partitions. We first take $t$ copies of each nonnegative integer $k$, denoted by $k_i$ for $i=0,\ldots, t-1$, and give the following total ordering:
\[
	k_i > m_j \text{ if and only if } k> m \text{ or } k=m, i>j. 
\]
We call $k_i$ the integer $k$ with color $i$. A $t$-colored Frobenius partition of $n$ is a two-rowed array, in which each row has nonnegative integers with possible $t$ colors in strictly decreasing order, and the sum of the numerical values of the entries plus the number of columns equals $n$. For example, the following are $2$-colored Frobenius partitions of $2$:
\[
    \begin{pmatrix}
        0_1 & 0_0\\ 0_1& 0_0
    \end{pmatrix},
    \begin{pmatrix}
        1_0\\ 0_0
    \end{pmatrix},
    \begin{pmatrix}
        1_1\\ 0_0
    \end{pmatrix},
    \begin{pmatrix}
        1_0\\ 0_1
    \end{pmatrix},
    \begin{pmatrix}
        1_1\\ 0_1
    \end{pmatrix},
    \begin{pmatrix}
        0_0\\ 1_0
    \end{pmatrix},
    \begin{pmatrix}
        0_1\\ 1_0
    \end{pmatrix},
    \begin{pmatrix}
        0_0\\ 1_1
    \end{pmatrix},
    \begin{pmatrix}
        0_1\\ 1_1
    \end{pmatrix}.
\]
For a partition $\lambda$, Kolitsch~\cite{K} introduced the {\it $t$-colored Frobenius partition} $\mathfrak{CF}_t(\lambda)$ of $\lambda$ as follows: for
\[
    \mathfrak{F}(\lambda)= \begin{pmatrix}
    	a_1& a_2 & \cdots & a_s\\ b_1 & b_2 & \cdots & b_s 
    \end{pmatrix},
\]
\[
    	\mathfrak{CF}_t(\lambda)=
	\begin{pmatrix} 
		{q_1}_{(r_1)} & {q_2}_{(r_2)} & \cdots & {q_s}_{(r_s)} \\ 
		{q'_1}_{(r'_1)} & {q'_2}_{(r'_2)} & \cdots & {q'_s}_{(r'_s)} 
	\end{pmatrix},
\]
where $a_i=tq_i+r_i$ and $b_i=tq'_i+(t-r'_i-1)$ for nonnegative integers $q_i$ and $r_i$ with $0 \le r_i,r'_i \le t-1$. Here, the subscripts $(r_i)$ and $(r'_i)$ denote colors. 
For example, the Frobenius partition and $3$-colored Frobenius partition of $\lambda=(8,7,7,4,4,2)$ are
\[
    	\mathfrak{F}(\lambda)=
	\begin{pmatrix} 
		7 & 5 & 4 & 0 \\
		5 & 4 & 2 & 1 
	\end{pmatrix}
    	\quad \text{and} \quad
    	\mathfrak{CF}_t(\lambda)=
	\begin{pmatrix} 
		2_{1} & 1_{2} & 1_{1} & 0_{0} \\
		1_{1} & 1_{0} & 0_{1} & 0_{0} 
	\end{pmatrix}.
\]
We note that the order of the entries corresponding to $b_i$'s with the same $q'_i$ is reversed in $\mathfrak{CF}_t(\lambda)$, i.e.,
\[
	tq+(t-r'_i-1)> tq+(t-r'_j-1)\quad  \text{ if and only if } \quad q_{r'_i} < q_{r'_j}.
\]

\subsection{Jacobi's triple product identity and Wright's map}\label{sec2.2}

The following triple product identity of Jacobi is instrumental in proving the Littlewood decomposition: 
\begin{equation}\label{JTP}
	(-zq;q)_{\infty}(-1/z;q)_{\infty} =\frac{1}{(q;q)_{\infty}}\sum_{n=-\infty}^{\infty} z^nq^{\binom{n+1}{2}}.
\end{equation}
This identity has several proofs, analytic or combinatorial. In our paper, we will use the combinatorial proof of Wright~\cite{Pak, Wright}. The infinite product on the left side of \eqref{JTP} is the generating function for two-rowed arrays of nonnegative integers:
\[
	\begin{pmatrix} 
		a_1& a_2 & \ldots & a_{u} \\ b_1 & b_2 & \ldots & b_{v}
	\end{pmatrix}.
\]
Note that $u$ is not necessarily equal to $v$. On the other hand, the right side of \eqref{JTP} generates pairs $(\Delta, \mu)$ of a staircase partition $\Delta$ and an ordinary partition $\mu$. A staircase partition is a partition with parts $1,2,\ldots, k$ only, i.e., $\Delta=(k,k-1,\ldots, 2,1)$ for some $k\geq 1$.  
Wright's map is a weight-preserving bijection between such two-rowed arrays and pairs $(\Delta, \mu)$. This can be explicitly defined as follows:
\[
    \begin{pmatrix} 
    	a_1& a_2 & \ldots & a_{u} \\ b_1 & b_2 & \ldots & b_{v}
    \end{pmatrix} \longmapsto \big(\Delta, \mu \big),
\]
where
\[
    \Delta= \begin{cases} (u-v,\ldots, 1) & \text{ if $u\ge v$},\\
    	(v-u-1,\ldots, 1) & \text{ if $u<v$},
    \end{cases}
    \qquad \text{and} \qquad
	\mu=(\mu_1,\dots,\mu_\ell)
\] 
with
$\mu_i=a_{i}+i-(u-v) \text{ for $i\le u$}$
and $(\mu_{u+1},\dots,\mu_{\ell})$ being the conjugate of a partition $\nu=(b_{1}-v+1,b_{2}-v+2,\dots,b_{v})$. We allow zero parts in the partition $\nu$. 

This map can be visualized using generalized Young diagrams (see \cite{Yee}). For convenience, we use dots instead of boxes. Fix a diagonal line from top left to bottom right. We first place $u$ dots on the diagonal, and then place $a_1,a_2,\dots,a_u$ dots to the right of each dot on the diagonal starting from top to bottom. Next, we place $b_v, b_{v-1},\dots, b_1$ dots below each dot on the diagonal starting from right to left. We then separate $\binom{u-v+1}{2}$ many dots to get $\Delta$ and the remaining dots form $\mu$.

For example, in Figure \ref{fig2.1}, the leftmost picture shows the arrangement of dots for the diagonal and the top row entries, and the arrangement for the bottom row entries is shown in the middle picture. The rightmost picture shows how to separate $\binom{u-v+1}{2}$ many dots.  
\begin{figure}[ht!]
    \begin{tikzpicture}[scale=1.5]
    
        \foreach \x in {0,...,4}
            \filldraw (\x*.25, -\x*0.25) circle (.5mm);
            
        \foreach \x in {1,...,6}
            \filldraw (\x*.25, 0) circle (.5mm);
        \foreach \x in {2,...,6}
            \filldraw (\x*.25, -.25) circle (.5mm);
        \foreach \x in {3,...,5}
            \filldraw (\x*.25, -.5) circle (.5mm);
        \foreach \x in {4,...,5}
            \filldraw (\x*.25, -.75) circle (.5mm);
        
        \foreach \x in {0,...,4}
            \filldraw (3+\x*.25, -\x*0.25) circle (.5mm);
            
        \foreach \x in {1,...,6}
            \filldraw (3+\x*.25, 0) circle (.5mm);
        \foreach \x in {2,...,6}
            \filldraw (3+\x*.25, -.25) circle (.5mm);
        \foreach \x in {3,...,5}
            \filldraw (3+\x*.25, -.5) circle (.5mm);
        \foreach \x in {4,...,5}
            \filldraw (3+\x*.25, -.75) circle (.5mm);
            
        \foreach \x in {3,...,6}
            \filldraw (3.5, -\x*0.25) circle (.5mm);
        \foreach \x in {4,...,5}
            \filldraw (3.75, -\x*.25) circle (.5mm);
        \foreach \x in {5,...,5}
            \filldraw (4, -\x*.25) circle (.5mm);
        \draw[thick, blue] (-0.2, 0.2) -- (1.2, -1.2);
        \draw[thick, blue] (2.8, 0.2) -- (4.2, -1.2);
        
        \foreach \x in {0,...,4}
            \filldraw (6+\x*.25, -\x*0.25) circle (.5mm);
            
        \foreach \x in {1,...,6}
            \filldraw (6+\x*.25, 0) circle (.5mm);
        \foreach \x in {2,...,6}
            \filldraw (6+\x*.25, -.25) circle (.5mm);
        \foreach \x in {3,...,5}
            \filldraw (6+\x*.25, -.5) circle (.5mm);
        \foreach \x in {4,...,5}
            \filldraw (6+\x*.25, -.75) circle (.5mm);
            
        \foreach \x in {3,...,6}
            \filldraw (6.5, -\x*0.25) circle (.5mm);
        \foreach \x in {4,...,5}
            \filldraw (6.75, -\x*.25) circle (.5mm);
        \foreach \x in {5,...,5}
            \filldraw (7, -\x*.25) circle (.5mm);
            
        \draw[thick, blue] (6.38, 0.1) -- (6.38, -1.7);
    \end{tikzpicture}
    \caption{$\begin{pmatrix} 6& 5& 3& 2& 0 \\ & & 4& 2& 1\end{pmatrix} \mapsto \big( (2,1), ( 5,5,4,4,3,3,1) \big)$}
    \label{fig2.1}
\end{figure}
The $u<v$ case is illustrated in Figure \ref{fig2.2}.
\begin{figure}[ht!]
    \begin{tikzpicture}[scale=1.5]
    
    \foreach \x in {2,...,6}
        \filldraw (\x*.25, -0.5) circle (.5mm);
    \foreach \x in {3,...,5}
        \filldraw (\x*.25, -.75) circle (.5mm);
    \foreach \x in {4,...,5}
        \filldraw (\x*.25, -1) circle (.5mm);
    
    \foreach \x in {2,...,6}
        \filldraw (3+\x*.25, -0.5) circle (.5mm);
    \foreach \x in {3,...,5}
        \filldraw (3+\x*.25, -.75) circle (.5mm);
    \foreach \x in {4,...,5}
        \filldraw (3+\x*.25, -1) circle (.5mm);
    
    \foreach \x in {1,...,6}
        \filldraw (3, -\x*.25) circle (.5mm);
    \foreach \x in {2,...,6}
        \filldraw (3.25, -\x*.25) circle (.5mm);
    \foreach \x in {3,...,5}
        \filldraw (3.5, -\x*.25) circle (.5mm);
    \foreach \x in {4,...,5}
        \filldraw (3.75, -\x*.25) circle (.5mm);
    
    \draw[thick, blue] (-0, 0) -- (1.2, -1.2);
    \draw[thick, blue] (2.9, 0.1) -- (4.2, -1.2);
    
    \foreach \x in {2,...,6}
        \filldraw (6+\x*.25, -0.5) circle (.5mm);
    \foreach \x in {3,...,5}
        \filldraw (6+\x*.25, -.75) circle (.5mm);
    \foreach \x in {4,...,5}
        \filldraw (6+\x*.25, -1) circle (.5mm);
    
    \foreach \x in {1,...,6}
        \filldraw (6, -\x*.25) circle (.5mm);
    \foreach \x in {2,...,6}
        \filldraw (6.25, -\x*.25) circle (.5mm);
    \foreach \x in {3,...,5}
        \filldraw (6.5, -\x*.25) circle (.5mm);
    \foreach \x in {4,...,5}
        \filldraw (6.75, -\x*.25) circle (.5mm);
    
    \draw[thick, blue] (5.8, -0.38) -- (7.7, -0.38);
    \end{tikzpicture}
    \caption{$\begin{pmatrix}
    & & 4 & 2& 1\\
    6& 5& 3& 2& 0
    \end{pmatrix} \mapsto \big( (1), (7,6,6,4,2) \big) $}
    \label{fig2.2}                                                                  
\end{figure}

\section{The Littlewood decomposition via Frobenius partitions}\label{sec3}

Let $p(n)$ be the number of partitions of $n$. The Littlewood decomposition implies that
\begin{align} \label{bijection1}
	\sum_{n\ge 0} p(n) q^{n} &= \frac{1}{(q^t;q^t)_{\infty}^t} \sum_{n\ge 0} c_t(n) q^n,\, 
\end{align}
where $c_t(n)$ denotes the number of $t$-core partitions of $n$. 

Using the constant term method and the Jacobi triple product identity, we can derive 
\begin{align}
	\sum_{n\ge 0} p(n) q^n &=[z^0] \prod_{j=1}^{t} (-zq^j;q^t)_{\infty} (-q^{t-j}/z;q^t)_{\infty} \notag \\
	&= \frac{1}{(q^t;q^t)^t_{\infty}}  [z^0] \prod_{j=1}^t \sum_{ k=-\infty}^{\infty} z^k q^{jk+t k(k-1)/2} \notag \\
	&=\frac{1}{(q^t;q^t)_{\infty}^t} \sum_{(m_1,\ldots, m_{t})\in \mathbb{Z}^t \atop m_1+\cdots +m_t=0} q^{ \sum_{j=1}^{t} j m_j + t  \sum_{j=1}^t m_j(m_j-1) /2}. \label{new-bijection1}
\end{align}
It is easy to check that the sum on the right side of \eqref{new-bijection1} is equal to the generating function for $c_t(n)$ in  \cite[Eq. (2.2)]{GKS}, given by
\begin{equation}\label{char}
	\sum_{n\ge 0} c_t(n)q^n = \sum_{\vec{n}\in \mathbb{Z}^t \atop \vec{n} \cdot \vec{1}=0}q^{\frac{t}{2}||\vec{n}||^2 + \vec{b} \cdot \vec{n}}, 
\end{equation}
where $\vec{b}=(0,1,\dots,t-1)$ and $\vec{1}=(1,1,\dots,1)$. Thus,
\begin{align}\label{bijection2}
	\sum_{n\ge 0} c_t(n)q^n  = \sum_{(m_1,\ldots, m_{t})\in \mathbb{Z}^t \atop m_1+\cdots +m_t=0} q^{t  \sum_{j=1}^t m_j(m_j-1) /2+ \sum_{j=1}^{t} j m_j},
\end{align}
from which with \eqref{new-bijection1}, we arrive at \eqref{bijection1}.
This derivation leads us to our Theorem~\ref{main2}. 

\subsection{Bijection $\varphi$}  In this section, we prove Theorem~\ref{main2} by constructing a bijection $\varphi$,
\begin{align*}
    	\varphi: \; & \mathcal{P} \to \mathcal{C}_t\times \mathcal{P}^t \\ 
	& \lambda \mapsto \big(\lambda^{(t)}, (\lambda_{(0)}, \ldots, \lambda_{(t-1)})\big).
\end{align*}
Given a partition $\lambda$, we take its corresponding $t$-colored Frobenius partition
\[
	\mathfrak{CF}_t(\lambda)=
	\begin{pmatrix} 
		{q_1}_{(r_1)} & {q_2}_{(r_2)} & \cdots & {q_s}_{(r_s)} \\ 
		{q'_1}_{(r'_1)} & {q'_2}_{(r'_2)} & \cdots & {q'_s}_{(r'_s)} 
	\end{pmatrix},
\]
and then divide the entries into $t$ groups according to their colors to form two-rowed arrays $\mathfrak{C}(\lambda_{(j)})$, i.e.,
\begin{align*}
	\mathfrak{C}(\lambda_{(j)})=\begin{pmatrix} a_{j,1}\; a_{j,2} \; \cdots \; a_{j,u_j} \\ b_{j,1}\; b_{j,2}\; \cdots \; b_{j,v_j} \end{pmatrix}  \text{ for $j=0,\ldots, t-1$},
\end{align*}
where $a_{j,i}$ and $b_{j,i}$ are the top and bottom entries of $\mathfrak{CF}_t(\lambda)$ with color $j$, respectively. Note that $u_j$ and $v_j$ may be different.

We first construct $\mathfrak{CF}_t(\lambda^{(t)})$ from $\mathfrak{C}(\lambda_{(j)})$'s. 
Let $u_j -v_j =w_j$. If $w_j>0$, add $(w_j-1)_{j},(w_j-2)_{j},\dots, 1_j, 0_{j}$ to the top row of $\mathfrak{CF}_t(\lambda^{(t)})$; if $w_j<0$, add $(-w_j-1)_{j},(-w_j-2)_{j},\dots,1_j, 0_{j}$ to the bottom row of $\mathfrak{CF}_t(\lambda^{(t)})$. If $w_j=0$, $\mathfrak{CF}_t(\lambda^{(t)})$ has no entry with color $j$.  We also note that $(w_0, w_1, \dots, w_{t-1})$ is indeed the characteristic vector $c(\lambda)=(n_0, n_1, \dots, n_{t-1})$ in \eqref{char}.   

For $\lambda_{(j)}$, we drop the color $j$ of each entry in $\mathfrak{C}(\lambda_{(j)})$ and apply Wright's map from Section~\ref{sec2.2}. We take $\lambda_{(j)}$ to be the resulting ordinary partition $\mu$. 

For example, let $t=3$ and $\lambda=(8,7,7,4,4,2)$. Then, 
\begin{align*}
	\mathfrak{F}(\lambda)=\begin{pmatrix} 7\;\,  5\;\,  4 \;\, 0  \\  5\;\,  4 \;\,  2\;\,  1\end{pmatrix}
	\quad \text{and} \quad
	\mathfrak{CF}_3(\lambda)=\begin{pmatrix} 2_{1} \;\,  1_{2}\;\,  1_{1} \;\, 0_{0}  \\  1_{1} \;\,  1_{0} \;\,  0_{1}\;\,  0_{0}\end{pmatrix}.
\end{align*}
Now, we split $\mathfrak{CF}_3(\lambda)$ into $3$ arrays:
\[
    \mathfrak{C}(\lambda_{(0)})= \begin{pmatrix} \;\;\;\; 0\\ 1\;\, 0 \end{pmatrix}, \quad \mathfrak{C}(\lambda_{(1)}) = \begin{pmatrix}  2 \;\,  1   \\  1\;\,   0\end{pmatrix}, \quad \mathfrak{C}(\lambda_{(2)}) = \begin{pmatrix} 1 \\ \;\,  \end{pmatrix}.
\] 
Applying Wright's map, we get
\[
    \lambda_{(0)}=(2), \; \lambda_{(1)}=(3,3), \text{ and } \lambda_{(2)}=(1).
\]
Also, since
\[
    \mathfrak{CF}_3(\lambda^{(3)})=\begin{pmatrix} 0_{2}\\ 0_{0}\end{pmatrix},
\]
we obtain
\[
    \lambda^{(3)}=(3,1,1)
\]
with its characteristic vector $c(\lambda)=(-1,0,1)$.

It is trivial from Wright's map that each $\lambda_{(j)}$ is an ordinary partition. Also, we can easily check that $\mathfrak{F} (\lambda^{(t)})$ satisfies the conditions in Theorem~\ref{main1}, which will be proved in Section \ref{sec3.2}.  

Next, we verify that
\begin{equation}\label{weight}
    |\lambda|=|\lambda^{(t)}|+t \sum\limits_{j=0}^{t-1} |\lambda_{(j)}|. 
\end{equation}
From the construction of $\mathfrak{C}(\lambda_{(j)})$ with $\sum\limits_{j=0}^{t-1}u_j=\sum\limits_{j=0}^{t-1}v_j=s$, the size of $\lambda$ is given by
\begin{align*}
    |\lambda|=\sum_{i=1}^{s}(q_i t+r_i + q'_i t + t-r'_i-1)+s 
    &=\sum_{j=0}^{t-1} \left\{ \sum_{k=1}^{u_j} (a_{j,k} t + j) + \sum_{k=1}^{v_j} (b_{j,k}t+t-j) \right\}\\
    &=\sum_{j=0}^{t-1} \left( t \sum_{k=1}^{u_j} a_{j,k} + t \sum_{k=1}^{v_j} b_{j,k}  + jw_j + t v_j \right).
\end{align*}
Also, we have that
\[
    |\lambda^{(t)}|=\sum\limits_{j=0}^{t-1} \left( \frac{ w_j(w_j -1)}{2}t + j w_j \right)
\]
and
\begin{align*}
    |\lambda_{(j)}|&=\sum_{k=1}^{u_j}(a_{j,k}+k-w_j)+\sum_{k=1}^{v_j}(b_{j,k}-k+1)
    =\sum_{k=1}^{u_j}a_{j,k}+\sum_{k=1}^{v_j}b_{j,k} -\frac{w_j(w_j-1)}{2} +v_j.
\end{align*}
Thus, \eqref{weight} holds. 

It is also clear from the construction that $\varphi$ is invertible because Wright's map is a bijection. 
Therefore, we complete the proof of Theorem \ref{main2}. We can also obtain (\ref{bijection2}) by replacing $m_{j+1}$ with $w_j$.

\subsection{Frobenius partition representation of $t$-core partitions }\label{sec3.2}
In this section, we prove Theorem~\ref{main1}. 
For a partition $\lambda$ with
\[
	\mathfrak{F}(\lambda)=\begin{pmatrix} a_1 \;\; a_2 \;\; \cdots a_s\\ b_1\;\; b_2\;\; \cdots b_s\end{pmatrix},
\]
we consider the hook of the box $(i,j)$ in three cases: 1) $i,j \leq s$, 2) $i\le s< j$, and 3) $j \leq s <i $.
We note that by the definition of $h_{i,j}(\lambda)$,
\begin{equation*}
        h_{i,j}:=h_{i,j}(\lambda)= (\lambda_i-i)+ (\lambda'_j -j)+1. 
\end{equation*}

\begin{enumerate}
	\item[Case 1:] $i,j\le s$. The hook length of the box $(i,j)$ is $h_{i,j}=a_i+b_j+1$. 
	For example,  given $\lambda=(8,7,7,4,4,2)$,  the hook length of the box $(1,2)$ is $12$, which equals $7+4+1$ as shown in Figure~\ref{fig:hooklength1} and  
    \begin{align*}
        \mathfrak{F}(\lambda)=\begin{pmatrix} 7\;\,  5\;\,  4 \;\, 0  \\  5\;\,  4 \;\,  2\;\,  1\end{pmatrix}. 
    \end{align*} 
    
    \begin{figure}[ht!]
        \centering	
        \ytableausetup{mathmode, boxsize=1em}
        \begin{ytableau}
             *(gray!30) & *(gray!30) &  *(gray!30)&  *(gray!30)&  *(gray!30)&  *(gray!30)& *(gray!30)&  *(gray!30)\\
                 &  & & & & &   \\
                 & *(gray!30) & & & & &  \\
                 & *(gray!30)  & &\\
                  & *(gray!30)  & & \\
                & *(gray!30) 
        \end{ytableau}
        \qquad $\longleftrightarrow$ \qquad 
        \begin{ytableau}
            *(white) & *(gray!30) &  *(gray!30)&  *(gray!30)&  *(gray!30)&  *(gray!30)& *(gray!30)&  *(gray!30)\\
                 & *(gray!30)  & & & & &   \\
                 & *(gray!30) & & & & &  \\
                 & *(gray!30)  & &\\
                  & *(gray!30)  & & \\
                & *(gray!30) 	
        \end{ytableau}
        \qquad $\longleftrightarrow$ \qquad 
        \begin{ytableau}
            *(white) &  & &  &  & & *(gray!30)&  *(gray!30)\\
                 &   & & & & &  *(gray!30)  \\
                 &  & &  *(gray!30)  & *(gray!30)   &  *(gray!30)  &   *(gray!30) \\
                 &  & &  *(gray!30)  \\
                  &   *(gray!30)   & *(gray!30)   &  *(gray!30)  \\
                & *(gray!30) 	
        \end{ytableau}	
        \caption{The hook of the box $(1,2)$ of the partition $(8,7, 7,4,4,2)$ }
        \label{fig:hooklength1} 
    \end{figure}
    
    \item[Case 2:] $i\le s<j$. Let $\ell=\lambda'_j$. As seen in Figure~\ref{fig:hooklength2}, $h_{i,j}$ is equal to the length of the rim hook from the last box of the $i$th row to the last box of the $j$th column. This rim hook consists of horizontal strips of boxes, where each but the last strip  has $a_k-a_{k+1}$ boxes for $k=i,\ldots, \ell-1$ and the length of the last strip is less than $a_{\ell}-a_{\ell+1}$. 
    We can now give the range for $h_{i,j}$ as follows:
    \begin{equation}\label{length2}
            a_i-a_\ell< h_{i,j}<  a_i-a_{\ell+1},  
    \end{equation}
    where $a_{s+1}=-1$. 
    
    \begin{figure}[ht!]
        \centering
        \begin{ytableau}
            *(white) & & & & &  *(gray!30)& *(gray!30)&  *(gray!30)\\
            &  & & &  &  *(gray!30)  &   \\
            &  & & & &*(gray!30)  &  \\
            &   & &\\
            &   & & \\
            & 
        \end{ytableau}
        \qquad $\longleftrightarrow$ \qquad 
        \begin{ytableau}
            *(white) & & & & &  & *(gray!30)&  *(gray!30)\\
            &  & & &  &   &  *(gray!30)  \\
            &  & & & & *(gray!30)   & *(gray!30)  \\
            &   & &\\
            &   & & \\
            & 
        \end{ytableau}
        \caption{The hook of the box $(1,6)$ of the partition $(8,7, 7,4,4,2)$ }
        \label{fig:hooklength2} 
    \end{figure}
    
    \item[Case 3:] $j\le s<i$. For the hook of the box $(i,j)$, the corresponding rim hook consists of vertical strips. Analogous to Case~2, we can compute the range of $h_{(i,j)}$, which is
    \begin{equation*}
            b_j - b_\ell< h_{i,j} < b_j-b_{\ell+1}, 
    \end{equation*}
    where $\ell=\lambda_i$ and $b_{s+1}=-1$.
\end{enumerate}
Based on this analysis, we characterize $t$-core partitions using Frobenius partitions.

\begin{proof}[Proof of Theorem \ref{main1}]
Let $\lambda$ be a partition with 
\[
	\mathfrak{F}(\lambda)= \begin{pmatrix} a_1 & a_2 &\cdots & a_s\\ b_1 & b_2 & \cdots & b_s\end{pmatrix}.
\]
The first condition that $a_i+b_j+1 \not \equiv 0 \pmod{t}$ immediately follows from Case~1. If $a_i+b_j+1\equiv 0 \pmod{t}$, then $h_{i,j}$ is a multiple of $t$, so $\lambda$ is not a $t$-core partition.

For the other conditions, suppose that $a_i=tq_i+r_i$ appears in the top row but $a_i-t=(t-1)q_i+r_i $ does not appear in the top row. Let $a_{k}=t q_k+ r_k$ be the smallest entry in the top row that is larger than $a_i-t$. Then $a_{k+1}=t q_{k+1}+ r_{k+1}$ is smaller than $a_i-t$ if $a_{k+1}$ exists, i.e.,
\begin{equation}\label{lengthrange}
	tq_i+r_i \ge t q_k  + r_k  > t (q_i-1)+r_i > t q_{k+1}+r_{k+1}. 
\end{equation}
The Frobenius partition is of the following form:
\[
	\mathfrak{F}(\lambda)=\begin{pmatrix} \cdots \;\; t q_i +r_i  \;\; \cdots \;\;  t q_{k} +r_k \;\; t q_{k+1}+ r_{k+1}  \;\; \cdots \\ \cdots \;\; \cdots \;\; \cdots \;\;  \cdots \;\; \cdots \;\; \cdots \;\; \cdots \;\; \cdots \;\; \cdots \;\; \cdots  \end{pmatrix}.
\]
By \eqref{length2}, the hook length $h_{i,j}$ ranges as follows:
\begin{equation} \label{11}
	t(q_i-q_k)+r_i-r_k < h_{i,j} < t(q_i-q_{k+1})+r_i-r_{k+1}. 
\end{equation}
Applying the inequality in \eqref{lengthrange}, we get 
\begin{equation}\label{12}
	t(q_i-q_k)+r_i-r_k \le t-1 
	\quad \text{ and } \quad 
	t+1 \le t(q_i-q_{k+1})+r_i-r_{k+1}. 
\end{equation}
It follows from \eqref{11} and \eqref{12} that $h_{i,j}=t$ for some $j$, i.e., there is a box of hook length $t$. This is a contradiction. 

We can similarly prove that if $b_{j}$ appears in the bottom row, then $b_{j}-t$ must appear in the bottom row. We omit the details. 
\end{proof}

Kolitsch \cite{K} gave the following characterization for $t$-core partitions in terms of $t$-colored Frobenius partitions, which is equivalent to our theorem. However, our proof does not use the biinfinite words from \cite{GKS}.

\begin{theorem}\cite[Theorem 1]{K}\label{thm3.1}
	A partition $\lambda$ is $t$-core if and only if $\mathfrak{CF}_t(\lambda)$ satisfies the conditions that no color appears in both rows and if $a_{k}$ appears in one row, then $(a-1)_{k},\ldots, 1_{k}, 0_{k}$ also appear in the same row. 
\end{theorem}

\subsection{The number of hooks of length $t$}
In this section, we consider hooks of length $t$ and prove the following theorem using Frobenius partitions. This result can be found in \cite[p. 468]{EC2}. 

\begin{theorem}
	The number of hooks of length $t$ in a partition $\lambda$ equals the number of hooks of length $1$ in its $t$-quotient $\big(\lambda_{(0)},\ldots, \lambda_{(t-1)}\big)$.
\end{theorem}
\begin{proof}
Let
\[
	\mathfrak{F}(\lambda)
	=\begin{pmatrix} a_1& a_2 & \cdots & a_s \\ b_1& b_2 & \cdots & b_s
	\end{pmatrix}.
\]
Suppose $h_{i,j}(\lambda)=t$ for some $i,j$. We consider three cases: 1) $i,j \leq s$, 2) $i\le s< j$, and 3) $j \leq s <i $ as discussed in the beginning of Section~\ref{sec3.2}. Recall that
\[
	a_i=t q_i+ r_i\qquad\text{and}\qquad b_j=t q'_j + (t - r'_j - 1).
\]

\begin{enumerate}
	\item[Case 1:] $i,j \leq s$. This means that
	\[
    		a_i+b_j+1=t,
    \]
    so $q_i=q'_j=0$ and $r'_j=r_i$. Thus, in $\mathfrak{CF}(\lambda)$, both $a_i$ and $b_j$ will appear as $0_{r_i}$ in the top and bottom rows. Hence, in $\mathfrak{C}(\lambda_{(r_i)})$, they are placed in the last column as follows:
    \[
        \mathfrak{C}(\lambda_{(r_i)}) = \begin{pmatrix} & \cdots & 0\\ & \cdots & 0\end{pmatrix},
    \]
    from which we see that the last column corresponds to the corner box in the Young diagram of $\lambda_{(r_i)}$ whose hook length equals $1$.  

    \item[Case 2:] $i\leq s < j $. Let $\ell=\lambda'_{j}$, i.e., the length of the $j$th column in the Young diagram of $\lambda$. 
     By \eqref{length2}, we have
     \[
     a_i>h_{i,j}+a_{\ell+1}. 
     \]
     Since $h_{i,j}=t$ and $a_{\ell+1}\ge -1$, we get
 	\[
 		a_i\ge t. 
 	\]
	Also, $a_i-t$ cannot appear in the top row of $\mathfrak{F}(\lambda)$. If it did appear, by the argument seen in the proof of Theorem~\ref{main1}, there would be no box $(i,j)$ of hook length $t$. This is a contraction. Thus, in $\mathfrak{C}\big(\lambda_{(r_i)}\big)$, $q_i\ge 1$ and the next entry to $q_i$ is at least 2 less than $q_i$ if exists. This means that there is no box below the last box of the row corresponding to $q_i$, so the hook of the last box is of length 1. 

	\item[Case 3:] $j\leq s < i $. Analogous to Case~2, we can show that if there is a hook of length $t$, it contributes to a hook of length $1$ in the $t$-quotient. We omit the details. 
\end{enumerate}
In each case, the correspondence between hooks of length $t$ in $\lambda$ and hooks of length $1$ in the $t$-quotient of $\lambda$ is indeed one-to-one. This completes the proof. 
\end{proof}

\section{The Littlewood decomposition of self-conjugate partitions and doubled distinct partitions}\label{sec4}

In this section, we apply the bijection $\varphi$ to self-conjugate partitions and doubled distinct partitions.

\subsection{Self-conjugate partitions}\label{sec:sc}
For a partition $\lambda$, recall that $\lambda'_j$ denotes the number of boxes in the $j$th column of the Young diagram of $\lambda$. The partition $\lambda'=(\lambda'_1,\lambda'_2,\dots,\lambda'_m)$ is called the conjugate of $\lambda$. If $\lambda=\lambda'$, then $\lambda$ is called a self-conjugate partition.

It is easy to check that the Frobenius partition of a self-conjugate partition $\lambda$ is of the form
\[
	\mathfrak{F}(\lambda)=
	\begin{pmatrix} 
		a_1 & a_2 & \cdots & a_s \\ 
		a_1 & a_2 & \cdots & a_s 
	\end{pmatrix},
\]
where $a_1>a_2>\cdots a_s \ge 0$. 
For example, the Frobenius partition of the self-conjugate partition $\lambda=(8,5,5,4,3,1,1,1)$ is
\[
	\mathfrak{F}(\lambda)=
	\begin{pmatrix} 
		7 & 3 & 2 & 0 \\ 
		7 & 3 & 2 & 0 
	\end{pmatrix}.
\]

Let $sc(n)$ be the number of self-conjugate partitions of $n$. 
Using Jacobi's triple product identity, we get its generating function.
If $t$ is even,
\begin{align*}
	\sum_{n\ge 0} sc(n) q^n &= \prod_{j=1}^{t/2} (-q^{2j-1};q^{2t})_{\infty} (-q^{2t-2j+1};q^{2t})_{\infty} \\
	&= \frac{1}{(q^{2t};q^{2t})^{\frac{t}2}_{\infty}}   \prod_{j=1}^{t/2}  \sum_{ k=-\infty}^{\infty} q^{(2j-1)k+t k(k-1)}\\
	&= \frac{1}{(q^{2t};q^{2t})^{\frac{t}2}_{\infty}}  \sum_{(m_1,\ldots, m_{t/2})\in \mathbb{Z}^{\frac{t}2}} q^{ \sum_{j=1}^{\frac{t}2}(2j-1) m_j + t  \sum_{j=0}^{\frac{t}2-1} m_j(m_j-1)}. 
\end{align*}
If $t$ is odd,
\begin{align*}
	\sum_{n\ge 0} sc(n) q^n &= (-q^{t};q^{2t})_{\infty} \prod_{j=1}^{(t-1)/2} (-q^{2j-1};q^{2t})_{\infty} (-q^{2t-2j+1};q^{2t})_{\infty} \\
	&= \frac{(-q^{t};q^{2t})_{\infty} }{(q^{2t};q^{2t})^{\frac{t-1}2}_{\infty}}  \sum_{(m_1,\ldots, m_{(t-1)/2})\in \mathbb{Z}^{\frac{t-1}2}} q^{ \sum_{j=1}^{\frac{t-1}2}(2j-1) m_j + t  \sum_{j=1}^{\frac{t-1}2} m_j(m_j-1)}. 
\end{align*}

The following result on self-conjugate partitions can be found in \cite{Olsson}. 

\begin{proposition}\cite[Proposition 3.5]{Olsson}
	For a self-conjugate partition $\lambda$, the Littlewood decomposition $\phi$ in \cite{JK} implies that its $t$-core $\lambda^{(t)}$ is also self-conjugate and the $t$-quotient satisfies $\lambda_{(j)}'=\lambda_{(t-j-1)}$ for $j=0,1,\dots, t-1$. 
\end{proposition}

We can prove this result using our bijection $\varphi$. Let $\lambda$ be a self-conjugate partition such that 
\[
	\mathfrak{F}(\lambda)=
	\begin{pmatrix} 
		a_1 & a_2 & \cdots & a_s \\ 
		a_1 & a_2 & \cdots & a_s 
	\end{pmatrix}
	\quad \text{and} \quad
	\mathfrak{CF}_t(\lambda)=
	\begin{pmatrix} 
		{q_1}_{(r_1)} & {q_2}_{(r_2)} & \cdots & {q_s}_{(r_s)} \\ 
		{q_1}_{(t-r_1-1)} & {q_2}_{(t-r_2-1)} & \cdots & {q_s}_{(t-r_s-1)}
	\end{pmatrix},
\]
where $a_j=tq_j+r_j$ for nonnegative integers $q_j$ and $r_j$ with $0 \le r_j \le t-1$. Then, we have
\begin{align*}
	\mathfrak{C}(\lambda_{(j)})=\begin{pmatrix} a_{j,1}\; a_{j,2} \; \cdots \; a_{j,u_j} \\ b_{j,1}\; b_{j,2}\; \cdots \; b_{j,v_j} \end{pmatrix}
	\quad \Leftrightarrow \quad
	\mathfrak{C}(\lambda_{(t-j-1)})=\begin{pmatrix} b_{j,1}\; b_{j,2}\; \cdots \; b_{j,v_j} \\ a_{j,1}\; a_{j,2} \; \cdots \; a_{j,u_j} \end{pmatrix}  
\end{align*}
for $j=0,\ldots, t-1$. Hence, the entries of the top row match the entries in the bottom row in $\mathfrak{F}(\lambda^{(t)})$, so $\lambda^{(t)}$ is self-conjugate and $\lambda_{(j)}'=\lambda_{(t-j-1)}$ for $j=0,1,\dots,t-1$.

For example, let $t=3$ and $\lambda=(8,5,5,4,3,1,1,1)$ be a self-conjugate partition. We have
\begin{align*}
	\mathfrak{F}(\lambda)=\begin{pmatrix} 7\;\,  3\;\,  2 \;\, 0  \\  7\;\,  3 \;\,  2\;\,  0\end{pmatrix}
	\quad \text{and} \quad
	\mathfrak{CF}_3(\lambda)=\begin{pmatrix} 2_{1} \;\,  1_{0}\;\,  0_{2} \;\, 0_{0}  \\  2_{1} \;\,  1_{2} \;\,  0_{2}\;\,  0_{0}\end{pmatrix}.
\end{align*}
We split $\mathfrak{CF}_3(\lambda)$ into 3 arrays:
\[
	\mathfrak{C}(\lambda_{(0)})= \begin{pmatrix} 1\;\; 0\\ ~\;\; 0 \end{pmatrix}, \quad \mathfrak{C}(\lambda_{(1)}) = \begin{pmatrix}  2   \\  2\end{pmatrix}, \quad \mathfrak{C}(\lambda_{(2)}) = \begin{pmatrix}  ~\;\; 0 \\ 1\;\; 0  \end{pmatrix},
\]
from which we get $\lambda^{(3)}=(1)$, $\lambda_{(0)}=(1,1)$, $\lambda_{(1)}=(3,1,1)$, and $\lambda_{(2)}=(2)$.

\subsection{Doubled-distinct partitions}\label{sec:dd}
For a partition $\lambda=(\lambda_1,\lambda_2,\dots,\lambda_s)$ into distinct parts, a doubled distinct partition $\lambda\lambda$ is a partition whose Frobenius partition is of the form
\[
	\mathfrak{F}(\lambda\lambda)=
	\begin{pmatrix} 
		\lambda_1 & \lambda_2 & \cdots & \lambda_s \\ 
		\lambda_1-1 & \lambda_2-1 & \cdots & \lambda_s-1 
	\end{pmatrix}.
\]
For example, for $\lambda=(8,4,3,1)$,
\[
	\mathfrak{F}(\lambda\lambda)=
	\begin{pmatrix} 
		8 & 4 & 3 & 1 \\ 
		7 & 3 & 2 & 0 
	\end{pmatrix}.
\]
Thus, $\lambda\lambda=(9,6,6,5,3,1,1,1)$.


Let $dd(n)$ be the number of doubled distinct partitions of $n$. We also find the generating function identity for doubled distinct partitions.
If $t$ is odd,
\begin{align*}
	\sum_{n\ge 0} dd(n) q^n &= (-q^{2t};q^{2t})_{\infty} \prod_{j=1}^{(t-1)/2} (-q^{2j};q^{2t})_{\infty} (-q^{2t-2j};q^{2t})_{\infty} \\
	&= \frac{(-q^{2t};q^{2t})_{\infty} }{(q^{2t};q^{2t})^{\frac{t-1}2}_{\infty}}  \sum_{(m_1,\ldots, m_{(t-1)/2})\in \mathbb{Z}^{\frac{t-1}2}} q^{ \sum_{j=1}^{\frac{t-1}2}(2j) m_j + t  \sum_{j=1}^{\frac{t-1}2} m_j(m_j-1)}. 
\end{align*}
If $t$ is even,
\begin{align*}
	\sum_{n\ge 0} dd(n) q^n &= (-q^{t};q^{2t})_{\infty} (-q^{2t};q^{2t})_{\infty} \prod_{j=1}^{t/2-1} (-q^{2j};q^{2t})_{\infty} (-q^{2t-2j};q^{2t})_{\infty} \\
	&= \frac{(-q^{t};q^{t})_{\infty} }{(q^{2t};q^{2t})^{\frac{t}2-1}_{\infty}}  \sum_{(m_1,\ldots, m_{t/2-1})\in \mathbb{Z}^{\frac{t}2-1}} q^{ \sum_{j=1}^{\frac{t}2-1}(2j) m_j + t  \sum_{j=1}^{\frac{t}2-1} m_j(m_j-1)}. 
\end{align*}

Garvan, Kim, and Stanton \cite{GKS} obtained the following result by restricting the Littlewood decomposition to doubled distinct partitions.

\begin{proposition}\cite[Bijection 3]{GKS}
     For a doubled distinct partition $\mu$, the Littlewood decomposition $\phi(\mu)=\big( \mu^{(t)},(\mu_{(0)},\mu_{(1)},\dots,\mu_{(t-1)})\big)$ implies that its $t$-core $\mu^{(t)}$ and $\mu_{(0)}$ are also doubled distinct and $\mu_{(j)}'=\mu_{(t-j)}$ for $j=1,\dots,t-1$. 
\end{proposition}

Similarly to Section \ref{sec:sc}, we can reprove this result using Frobenius partitions.
Let $\mu$ be a doubled distinct partition with 
\[
	\mathfrak{F}(\mu)=
	\begin{pmatrix} 
		a_1 & a_2 & \cdots & a_s \\ 
		a_1-1 & a_2-1 & \cdots & a_s-1 
	\end{pmatrix},
\]
where $a_1>a_2>\cdots>a_s\ge 1$.
We first note that a column of $\mathfrak{F}(\mu)$ is represented in $\mathfrak{CF}(\mu)$ as follows: for some $q\ge 0$ and $1\le r\le t-1$,
\[
	\binom{tq}{tq-1} \in \mathfrak{F}(\mu) \longrightarrow \binom{q_0}{(q-1)_{0}} \in \mathfrak{CF}(\mu) \quad \text{and} \quad \binom{tq+r}{tq+r-1} \in \mathfrak{F}(\mu) \longrightarrow \binom{q_r}{q_{t-r}} \in \mathfrak{CF}(\mu).
\]

Thus, we have 	
\[
	\mathfrak{C}(\mu_{(0)})=\begin{pmatrix} \, a_{0,1}\phantom{-1}\; \;\;\; \,\, a_{0,2}\phantom{-1} \; \;\;\cdots \; \,\, \;\; a_{0,u_0}\phantom{-1} \\ a_{0,1}-1\;\;\;\;  a_{0,2}-1\;\;\;  \cdots \;\;\; a_{0,u_0}-1 \end{pmatrix},
\] 
and for $j=1,\ldots, t-1$,
\begin{align*}
	\mathfrak{C}(\mu_{(j)})=\begin{pmatrix} a_{j,1}\; a_{j,2} \; \cdots \; a_{j,u_j} \\ b_{j,1}\; b_{j,2}\; \cdots \; b_{j,v_j} \end{pmatrix}
	\quad \Leftrightarrow \quad
	\mathfrak{C}(\mu_{(t-j)})=\begin{pmatrix} b_{j,1}\; b_{j,2}\; \cdots \; b_{j,v_j} \\ a_{j,1}\; a_{j,2} \; \cdots \; a_{j,u_j} \end{pmatrix}.  
\end{align*}	
Therefore, we conclude that $\mu^{(t)}$ and $\mu_{(0)}$ are doubled distinct and $\mu_{(j)}'=\mu_{(t-j)}$ for $j=1,2,\dots,t-1$.	

For example, let $t=3$ and $\mu=(9,6,6,5,3,1,1,1)$ be a doubled distinct partition. We have
\begin{align*}
	\mathfrak{F}(\mu)=\begin{pmatrix} 8\;\,  4\;\,  3 \;\, 1  \\  7\;\,  3 \;\,  2\;\,  0\end{pmatrix}
	\quad \text{and} \quad
	\mathfrak{CF}_3(\mu)=\begin{pmatrix} 2_{2} \;\,  1_{1}\;\,  1_{0} \;\, 0_{1}  \\  2_{1} \;\,  1_{2} \;\,  0_{2}\;\,  0_{0}\end{pmatrix}.
\end{align*}
We split $\mathfrak{CF}_3(\mu)$ into 3 arrays:
\[
	\mathfrak{C}(\mu_{(0)})= \begin{pmatrix} 1 \\ 0 \end{pmatrix}, \quad \mathfrak{C}(\lambda_{(1)}) = \begin{pmatrix}  1\;\; 0 \\ ~\;\; 2\end{pmatrix}, \quad \mathfrak{C}(\lambda_{(2)}) = \begin{pmatrix}  ~\;\; 2 \\ 1\;\; 0  \end{pmatrix}.
\]
Therefore, we have $\mu^{(3)}=(2)$, $\mu_{(0)}=(2)$, $\mu_{(1)}=(1,1,1,1)$, and $\mu_{(2)}=(4)$.

\section*{Acknowledgements}
The authors would like to thank the anonymous reviewer for the careful reading and helpful comments, which have significantly improved our paper.  Hyunsoo Cho was supported by the National Research Foundation of Korea (NRF) grant (2021R1C1C2007589, 2019R1A6A1A11051177). Eunmi Kim was supported by the National Research Foundation of Korea grant (RS–2023-00244423, NRF–2019R1A6A1A11051177). Ae Ja Yee was partially supported by a grant ($\#633963$) from the Simons Foundation.



\begin{thebibliography}{99}

\bibitem{GEA}
G. E. Andrews, {\it Generalized Frobenius partitions}, Mem.~Amer.~Math.~Soc. {\bf 49}(1984), no. 301.

\bibitem{BD}
A. Berkovich and A. Dhar, {\it On partitions with bounded largest part and fixed integral GBG-rank modulo primes}, Ann. Comb. (2024). https://doi.org/10.1007/s00026-024-00733-y.

\bibitem{BG0}
A. Berkovich and  F.~G. Garvan, {\it On the Andrews-Stanley refinement of Ramanujan's partition congruence modulo 5 and generalizations}, Trans. Amer. Math. Soc. {\bf 358} (2006), no. 2, 703--726.

\bibitem{BG1}
A. Berkovich and  F. G. Garvan, {\it The BG-rank of a partition and its applications}, Adv. in Appl. Math. {\bf 40} (2008), no. 3, 377--400.

\bibitem{BG2}
A. Berkovich and  F. G. Garvan, {\it The GBG-rank and t-cores I. Counting and $4$-cores}, J. Comb. Number Theory {\bf 1} (2009), no. 3, 237--252.

\bibitem{BN} 
O. Brunat and R. Nath, {\it Cores and quotients of partitions through the Frobenius symbol}, preprint, arXiv:1911.12098.

\bibitem{CKKY}
H. Cho, E. Kim, B. Kim, and A. J. Yee, {\it On the distribution of $t$-hooks of doubled distinct partitions}, preprint, arxiv:2503.12040.

\bibitem{CKLLYY}
H. Cho, E. Kim, H.-H. Lee, K. Lee, A. J. Yee, and J. Yoo, {\it GBG-rank generating functions for ordinary, self-conjugate, and doubled distinct partitions}, preprint.  

\bibitem{GKS}
F. G. Garvan, D. Kim, and D. Stanton, {\it Cranks and $t$-core}, Invent. Math. {\bf 101} (1990), no. 1, 1--17.

\bibitem{James}
G. D. James, {\it Some combinatorial results involving Young diagrams}, Math. Proc. Cambridge Philos. Soc. {\bf 83} (1978), no. 1, 1--10.

\bibitem{JK}
G. James and A. Kerber, {\it The representation Theory of the Symmetric Group}, Encyclopedia of Mathematics and its Applications, vol. 16, Addison-Wesley Publishing Co., Reading, MA, 1981. 

\bibitem{K}
L. Kolitsch, {\it Generalized Frobenius partitions, $k$-cores, $k$-quotients, and cranks}, Acta Arith.  {\bf 62} (1992), no. 1, 97--102.

\bibitem{L}
D. E. Littlewood,  {\it Modular representations of symmetric groups}, Proc. Roy. Soc. London Ser. A {\bf 209} (1951), 333--353.

\bibitem{M}
I. Macdonald, {\it Symmetric functions and Hall polynomials. Second edition. With contributions by A. Zelevinsky}, Oxford Mathematical Monographs. Oxford Science Publications. The Clarendon Press, Oxford University Press, New York, 1995. 
 
\bibitem{N}
T. Nakayama, {\it On some modular properties of irreducible representations of a symmetric group. I.} Jpn. J. Math {\bf 17} (1941), 165--184.

\bibitem{Olsson}
J. B. Olsson, {\it Combinatorics and representations of finite groups}, Vorlesungen aus dem Fachbereich Mathematik der Universit\"{a}t GH 
Essen [Lecture Notes in Mathematics at the University of Essen], vol. 20, Universit\"{a}t Essen, Fachbereich Mathematik, Essen, 1993.

\bibitem{Pak}
I. Pak, {\it Partition bijections, a survey}, Ramanujan J. {\bf 12} (2006), no. 1, 5--75.

\bibitem{EC2}
R. P. Stanley, {\it Enumerative Combinatorics, Vol. 2}, Cambridge Studies in Advanced Mathematics, vol. 62, Cambridge University Press, Cambridge, 1999.

\bibitem{Wright}
E. M. Wright, {\it An enumerative proof of an identity of {J}acobi}, J. London Math. Soc. {\bf 40} (1965), 55--57.

\bibitem{Yee}
A. J. Yee, {\it Combinatorial proofs of generating function identities for F-partitions}, J. Combin. Theory Ser. A {\bf 102} (2003), no. 1, 217--228.

\end{thebibliography}
\end{document}